\documentclass[12pt]{amsart}
\usepackage{amssymb}
\usepackage[all]{xy}

\usepackage{pgf, tikz}
\definecolor {processblue}{cmyk}{0.96,0,0,0}
\input xy
\xyoption{all}
\usepackage{mathtools}
\usepackage{amsfonts}
\usepackage{amsthm}

\usepackage{amscd}
  \newtheorem{The}{Theorem}[section]
  \newtheorem{Pro}[The]{Proposition}
  \newtheorem{Lem}[The]{Lemma}
  \newtheorem{Cor}[The]{Corollary}
  
  \newtheorem{Rem}[The]{Remark}
  
  \newtheorem{Examp}[The]{Example}

\newcommand{\bsm}{\begin{smallmatrix}}
\newcommand{\esm}{\end{smallmatrix}}
\newcommand{\bbm}{\begin{matrix}}
\newcommand{\ebm}{\end{matrix}}

\begin{document}

\title{ Connections between Representation-Finite and K\"othe Rings}

\author{Ziba Fazelpour}
\address{ Department of Mathematics\\
University of Isfahan\\
P.O. Box: 81746-73441, Isfahan, Iran\\ and School of Mathematics, Institute for Research in Fundamental Sciences (IPM), P.O. Box: 19395-5746, Tehran, Iran}
\email{z.fazelpour@ipm.ir}
\author{Alireza Nasr-Isfahani}
\address{Department of Mathematics\\
University of Isfahan\\
P.O. Box: 81746-73441, Isfahan, Iran\\ and School of Mathematics, Institute for Research in Fundamental Sciences (IPM), P.O. Box: 19395-5746, Tehran, Iran}
\email{nasr$_{-}$a@sci.ui.ac.ir / nasr@ipm.ir}



\begin{abstract} A ring $R$ is called left $k$-cyclic if every left $R$-module is a direct sum of indecomposable modules which are homomorphic image of $_{R}R^k$. In this paper, we give a characterization of left $k$-cyclic rings. As a consequence, we give a characterization of left K\"othe rings, which is a generalization of K\"othe-Cohen-Kaplansky theorem. We also characterize rings which are Morita equivalent to a basic left $k$-cyclic ring. As a corollary, we show that $R$ is Morita equivalent to a basic left K\"othe ring if and only if $R$ is an artinian left multiplicity-free top ring.

\end{abstract}

\maketitle

\section{Introduction}
Let $R$ be an associative ring with unit. $R$ is called left pure semisimple if every
left $R$-module is a direct sum of finitely generated  left $R$-modules. $R$ is said to be
representation-finite if it is left artinian and has only finitely many non-isomorphic
finitely generated indecomposable left $R$-modules. According to Auslander \cite{Auslander}, Ringel and Tachikawa \cite{RT} and Fuller and Reiten \cite{FR}, a ring $R$ is representation-finite if and only if it is left and right pure semisimple. The pure semisimplicity conjecture, which says that left pure semisimple rings are representation-finite, is still open.

Recall that every finitely generated (abelian group) $\mathbb{Z}$-module is a direct sum of cyclic modules.
K\"{o}the proved that  artinian principal ideal rings have this property \cite{K}. A ring $R$ is called {\it left} (resp., {\it right}) {\it K\"{o}the} if every  left (resp., right) $R$-module is a direct sum of cyclic modules. K\"{o}the also posed the problem, known as K\"othe's problem, to classify the left (resp., right) K\"{o}the rings. The K\"othe's problem is still open.
Nakayama gave an example of a right K\"{o}the ring $R$ which is not a principal right ideal ring \cite[page 289]{Na}. Later, Cohen and Kaplansky proved that if a commutative ring $R$ is K\"othe, then $R$ is an artinian principal ideal ring \cite{CK}. In 1961, Kawada completely solved the K\"othe problem for the basic finite dimensional $K$-algebras \cite{Kawada1, Kawada2, Kawada3} (see also \cite{Ringel2}). Kawada's papers contain a set of 19 conditions which characterize Kawada algebras, as well as, the list of all possible finitely generated indecomposable modules. Ringel showed that any finite dimensional $K$-algebra of finite representation type is Morita equivalent to a K\"othe algebra. By using the multiplicity-free of top and soc of finitely generated indecomposable  modules, he also gave a characterization of Kawada algebras \cite{Ringel2}. Behboodi et al. proved that if  $R$ is a left K\"othe ring in which all idempotents are central, then $R$ is an artinian principal right ideal ring \cite{BGMS}. Recently Ghorbani et al.  gave a characterization of pure semisimple rings via $\rho$-dimension \cite{GNN}. They proved that for a ring $R$ and an epic class $\rho$ of indecomposable $R$-modules, $R$ is left and right pure semisimple if and only if every (left) right module is a direct sum of modules of $\rho$-dimension at most $n$ for some $n$, if and only if, ${\rm Mat}_n(R)$ is a left K\"othe ring for some positive integer $n$. In this paper,  we will continue the study of left K\"othe rings.

A ring $R$ is called {\it left $k$-cyclic} if every left $R$-module is a direct sum of indecomposable $k$-generated modules. In this paper, we first give a characterization of left $k$-cyclic rings. We show that $R$ is a left $k$-cyclic ring if and only if  $R$ is a left artinian ring and  for each finitely generated indecomposable  left $R$-module $M$, $c_i(top(M)) \leq kp_R(i)$ for each $1\leq i \leq m$, where $\lbrace e_1, \cdots, e_m\rbrace$ is a basic set of idempotents of $R$, ${R\cong \bigoplus_{i=1}^m {(Re_i)}^{p_R(i)}}$ and $c_i(top(M))$ is the number of composition factors of $top(M)$ which are isomorphic to the $Re_i/Je_i$ (see Theorem \ref{A7}). As a corollary, we give a characterization of left K\"othe rings (see Corollary \ref{A6}). In fact,  this corollary is a generalization of the K\"othe-Cohen-Kaplansky theorem and Theorem 3.1 of \cite{BGMS}. We also give a characterization of rings which are Morita equivalent to a basic left  $k$-cyclic ring. We show that a ring $R$ is Morita equivalent to a basic left K\"othe ring if and only if $R$ is an artinian left multiplicity-free top ring.
It is known that any left K\"othe ring is representation-finite. Finally, we show that a ring $R$ is representation-finite if and only if there exists a basic ring $S$ and a positive integer $n$ such that ${\rm Mat}_n(S)$ is a left K\"othe ring and $R$ is Morita equivalent to ${\rm Mat}_n(S)$ if and only if  $R$ is a left $n$-cyclic ring  for some positive integer $n$ if and only if ${\rm{Mat}}_m(R)$ is  a left K\"othe ring for some positive integer $m$.  Our results in this paper generalize and unify some results of \cite{BGMS}, \cite{CK}, \cite{GNN}, \cite{K} and \cite{Ringel2}.

The paper is organized as follows. In Section 2, we prove some preliminary results that will be needed later in the paper. In Section 3, we give a characterization of left $k$-cyclic rings and then we show that all known results about the K\"othe's problem are just corollary of this characterization. Finally in Section 4, we characterize the class of rings which are Morita equivalent to the given left K\"othe ring.
\subsection{Notation }
Throughout this paper, all rings have identity elements and all modules are unital. Let $R$ be a ring. We denote by $R$-Mod (resp., Mod-$R$) the category of all left (resp., right) $R$-modules and by $J$ the Jacobson radical of $R$. For a left $R$-module $M$, we denote by ${\rm soc}(M)$, ${\rm top}(M)$, ${\rm rad}(M)$ and $\ell(M)$ its socle, top, radical and length, respectively. Let $C$ and $D$ be two categories. We write $C \approx D$ in case $C$ and $D$ are equivalent. When two rings $R$ and $S$ are Morita equivalent (i.e., categories $R$-Mod and $S$-Mod are equivalent) we write $R \approx S$.

\section{Preliminaries}
A ring $R$ is called {\it semiperfect} if $R/J$ is a semisimple ring and idempotents lift modulo $J$. We recall that a set $\lbrace e_1, \cdots, e_m\rbrace$ of idempotents of a semiperfect ring $R$ is called {\it basic} in case they are pairwise orthogonal, $Re_i\ncong Re_j$ for each $i\neq j$  and for each indecomposable projective left $R$-module $P$, there exists $i$ such that $P \cong Re_i$. Clearly, the cardinal of any two basic sets of idempotents of a semiperfect ring $R$ are equal.  An idempotent $e$ of a semiperfect ring $R$ is called {\it basic idempotent} in case $e$ is the sum $e= e_1 + \cdots + e_m$ of a basic set $\lbrace e_1, \cdots, e_m\rbrace$ of idempotents of $R$. A semiperfect ring $R$ is called {\it basic} if $1_R$ is a basic idempotent (see \cite{Anderson}).\\

 Let $\lbrace e_1, \cdots, e_m\rbrace$ be a basic set of idempotents of a semiperfect ring $R$, $M$ be a left $R$-module of finite length and $1 \leq i \leq m$, we denote by $c_i(M)$, the number of composition factors of $M$ which are isomorphic to the $Re_i/Je_i$. Recall that in \cite{Warfield}, $c_i({\rm top}(M))$ is denoted by ${\rm Gen}(M, Re_i/Je_i)$.\\

The following proposition is a generalization of \cite[Lemma 1.8]{Warfield}.

 \begin{Pro}\label{A4}
 Let $R$ be a semiperfect ring, $M$ be a finitely generated left $R$-module and $k \in {\Bbb{N}}$. Then the following conditions are equivalent.
 \begin{itemize}
 \item[$(1)$] $M$ is a $k$-generated left $R$-module.
 \item[$(2)$] $c_i({\rm top}(M))  \leq kp_R(i)$ for each $1 \leq i \leq m$, where $\lbrace e_1, \cdots, e_m\rbrace$ is a basic set of idempotents of $R$ and ${R\cong \bigoplus_{i=1}^m {(Re_i)}^{p_R(i)}}$.
 \end{itemize}
 \end{Pro}
 \begin{proof}
 Assume that $R$ is a semiperfect ring and $M$ is a finitely generated left $R$-module. Then  ${R\cong \bigoplus_{i=1}^m {(Re_i)}^{p_R(i)}}$, where $m \in {\Bbb{N}}$, each $p_R(i) \in {\Bbb{N}}$ and $\lbrace e_1, \cdots, e_m \rbrace $ is a basic set of idempotents of $R$ and so by \cite[Corollary 15.18, Proposition 27.10]{Anderson}, ${\rm top}(M) \cong {(Re_1/Je_1)}^{s_1} \oplus \cdots \oplus {(Re_m/Je_m)}^{s_m}$. Thus there exists a projective cover $\rho: P(M) \rightarrow M$ of $M$, where $P(M) = {(Re_1)}^{s_1} \oplus \cdots \oplus {(Re_m)}^{s_m}$. Now assume that  $M$ is a $k$-generated left $R$-module. Then there exists an epimorphism $f: R^{(k)} \rightarrow M$. Therefore there exists a morphism $g:R^{(k)} \rightarrow P(M)$ such that ${\rho g = f}$.  Since ${\rm Ker}(\rho)$ is a superfluous submodule of $P(M)$,   $g$ is an epimorphism and hence  $P(M)$ is a direct summand of $R^{(k)}$. It follows that  $c_i({\rm top}(M)) \leq kp_R(i)$ for each $1 \leq i \leq m$.
Now assume that for each $1 \leq i \leq m$,  $c_i({\rm top}(M))  \leq kp_R(i)$. Then  $ s_i \leq kp_R(i)$ and so $P(M)$ is a direct summand of $R^{(k)}$. Therefore $M$ is a $k$-generated left $R$-module.
 \end{proof}

 One of the considerable properties of Morita theory which we use it in this paper is the fact that submodule lattices are preserved by equivalences. Let $M$ be a left $R$-module and $K$ be a submodule of $M$. We denote by $i_{K \leq M}: K \rightarrow M$ the inclusion monomorphism.

 \begin{Lem}\label{A2}
 Let $R$ and $S$ be Morita equivalent rings via an equivalence $F: R{\rm -Mod} \rightarrow {S{\rm -Mod}}$ and $M$ be a left $R$-module. If  ${\rm top}(M)$ is a semisimple left $R$-module, then $F({\rm top}(M)) \cong {\rm top}(F(M))$ as left $S$-modules.
 \end{Lem}
 \begin{proof}
  Assume that ${\mathcal{L}}_M$ (resp., ${\mathcal{L}}_{F(M)}$) is the lattice of submodules of $M$ (resp., $F(M)$). We define the map  ${\Lambda_{M}: {\mathcal{L}}_M \rightarrow {\mathcal{L}}_{F(M)}}$ with  ${\Lambda_{M}(K) = Im F(i_{K \leq M})}$ for each $K \in {\mathcal{L}}_M$. Then by \cite[Proposition 21.7]{Anderson}, ${\Lambda_{M}}$ is a lattice isomorphism. Let  ${\rm top}(M)$ be a semisimple left $R$-module. We have an exact sequence\vspace{-2mm}
 \begin{center}
 $0 \longrightarrow {\rm rad}(M) \overset{i_{{\rm rad}(M) \leq M}}{\longrightarrow} M \longrightarrow {\rm top}(M) \longrightarrow 0$.
 \end{center}
  Then
 $0 \longrightarrow F({\rm rad}(M)) \overset{F(i_{{\rm rad}(M) \leq M})}{\longrightarrow} F(M) \longrightarrow F({\rm top}(M))\longrightarrow 0$
  is an exact sequence. Thus  $F(M)/\Lambda_M({\rm rad}(M)) \cong F({\rm top}(M))$. Since $F(M)/\Lambda_M({\rm rad}(M))$ is a semisimple left $S$-module, ${\rm rad}(F(M)) \subseteq \Lambda_M({\rm rad}(M))$. On the other hand, by \cite[Proposition 21.7]{Anderson} and the fact that a submodule $K$ of $M$ is maximal if and only if $M/K$ is a simple module, we have $\Lambda_M({\rm rad}(M)) \subseteq {\rm rad}(F(M))$. Therefore $F({\rm top}(M)) \cong {\rm top}(F(M))$ as left $S$-modules.
 \end{proof}

 \begin{Lem}\label{A1}
Let $R$ be a semiperfect ring and $M$ be a  finite length  left $R$-module. Then $\ell\left( _{{\rm End}_R(Re_i)}{{\rm Hom}_R(Re_i, M)}\right) = c_i(M)$ for each $1 \leq i \leq m$, where $\lbrace e_1, \cdots, e_m\rbrace$ is a basic set of idempotents of $R$.
\end{Lem}
\begin{proof}
 Assume that $R$ is a semiperfect ring. Then  ${R\cong \bigoplus_{i=1}^m {(Re_i)}^{p_R(i)}}$, where $m \in {\Bbb{N}}$, each $p_R(i) \in {\Bbb{N}}$ and $\lbrace e_1, \cdots, e_m\rbrace$ is a basic set of idempotents of $R$.  Let $1 \leq i \leq m$. We  show that ${\rm Hom}_R(Re_i, Re_i/Je_i)$ is a simple left ${\rm End}_R(Re_i)$-module. Let $0\neq \alpha, \beta \in {\rm Hom}_R(Re_i, Re_i/Je_i)$. Then  there exists $f\in {\rm End}_R(Re_i)$ such that $\alpha f =\beta$. Therefore ${\rm Hom}_R(Re_i, Re_i/Je_i)$ is a simple left ${\rm End}_R(Re_i)$-module.
Now we show that for each $j\neq i$, ${\rm Hom}_R(Re_i, Re_j/Je_j)=0$. Let $i \neq j$ and $0\neq \gamma \in {\rm Hom}_R(Re_i, Re_j/Je_j)$. Then  $\gamma$ is an epimorphism and so ${\rm Ker}(\gamma)$ is a maximal submodule of $Re_i$. Therefore by \cite[Proposition 27.10]{Anderson}, $i=j$ which is a contradiction.  Let $M$ be a finite length left $R$-module. Then by the above argument and  \cite[Proposition 32.4]{Wi}, ${\rm Hom}(Re_i, M)$ is a finite length left ${\rm End}_R(Re_i)$-module. Now by induction on length of $M$ we show that  $\ell\left( _{{\rm End}_R(Re_i)}{{\rm Hom}_R(Re_i, M)}\right) = c_i(M)$ for each $1 \leq i \leq m$. Assume that $M$ is a simple left $R$-module. Then by \cite[Proposition 27.10]{Anderson}, there exists $1 \leq j \leq m$ such that $M \cong Re_j/Je_j$. Thus ${\rm Hom}_R(Re_i, M) \cong {\rm Hom}_R(Re_i, Re_j/Je_j)$ and so
$\ell\left( _{{\rm End}_R(Re_i)}{{\rm Hom}_R(Re_i, M)}\right) = \ell\left( _{{\rm End}_R(Re_i)}{{\rm Hom}_R(Re_i, Re_j/Je_j)}\right)$.
 Hence by the above argument,  ${\ell\left( _{{\rm End}_R(Re_i)}{{\rm Hom}_R(Re_i, M)}\right) = c_i(M)}$.
 Now assume that  $\ell(M) =t >1$ and   $0=M_t\subset M_{t-1} \subset \cdots \subset M_1 \subset M$ is a composition series for $M$. Then $0=M_{t-1}/M_{t-1} \subset M_{t-2}/M_{t-1}\subset \cdots \subset M_1/M_{t-1} \subset M/M_{t-1}$ is a composition series for  $M/M_{t-1}$. Therefore by the induction,  $\ell\left( _{{\rm End}_R(Re_i)}{{\rm Hom}_R(Re_i, M/M_{t-1})}\right) = c_i(M/M_{t-1})$. We consider the exact sequence  $0 \rightarrow M_{t-1} \rightarrow M \rightarrow M/M_{t-1} \rightarrow 0$. Then $0 \rightarrow {\rm Hom}_R(Re_i, M_{t-1}) \rightarrow {\rm Hom}_R(Re_i, M) \rightarrow {\rm Hom}_R(Re_i, M/M_{t-1}) \rightarrow 0$ is an exact sequence. So
 \begin{center}
 $\ell\left( _{{\rm End}_R(Re_i)}{{\rm Hom}_R(Re_i, M)}\right) = \ell\left( _{{\rm End}_R(Re_i)}{{\rm Hom}_R(Re_i, M_{t-1})}\right) + c_i(M/M_{t-1})$.
 \end{center}
If ${M_{t-1} \ncong Re_i/Je_i}$, then ${c}_i(M/M_{t-1}) = {c}_i(M)$ and  ${\rm Hom}_R(Re_i, M_{t-1}) = 0$. It follows that\\
$\ell\left( _{{\rm End}_R(Re_i)}{{\rm Hom}_R(Re_i, M)}\right) = {c}_i(M)$. Now assume that  ${M_{t-1} \cong Re_i/Je_i}$. Then  by the above argument,  ${\rm Hom}_R(Re_i, M_{t-1})$ is a simple  left ${\rm End}_R(Re_i)$-module. Consequently, $\ell\left( _{{\rm End}_R(Re_i)}{{\rm Hom}_R(Re_i, M)}\right) = {c}_i(M)$.
\end{proof}

Let $\lbrace M_1, \cdots, M_t \rbrace$ be the complete set of non-isomorphic finitely generated indecomposable left $R$-modules  and $\lbrace e_1, \cdots ,e_r \rbrace$ be a basic set of idempotents of $R$. For each $1 \leq l \leq r$, put
$q_R(l) = {\rm max} \lbrace c_l({\rm top}(M_j)) ~|~ 1 \leq  j \leq t \rbrace$.

 \begin{Pro}\label{A5}
 Let $R$ be a representation-finite ring which is Morita equivalent to a ring $S$. Then $r$ is the cardinal number of the basic set of idempotents of $S$ and $q_R(l) = q_S(l)$ for each $1\leq l \leq r$, where $r$ is the cardinal number of the basic set of idempotents of $R$.
 \end{Pro}
 \begin{proof}
Assume that  $R$ is a representation-finite ring. Then there exists a basic set of idempotents of $R$. Let  $R$ be Morita equivalent to a ring $S$  via an equivalence $F: R{\rm -Mod} \rightarrow {S{\rm -Mod}}$ and $\lbrace e_1,\cdots, e_r \rbrace$ be a basic set of idempotents of $R$. Then there exists a basic set of idempotents of $S$ and we show that $r$ is the cardinal number of the basic set of idempotents of $S$. It is enough to show that  $r$ is the cardinal number of the complete set of non-isomorphic  finitely generated indecomposable projective left $S$-modules. We know that $\lbrace F(Re_1), \cdots, F(Re_r)\rbrace$ is a set of  non-isomorphic  finitely generated indecomposable projective left $S$-modules. Let  $G:S{\rm -Mod} \rightarrow R{\rm -Mod}$ be the inverse equivalence of $F$ and $Q$ be a finitely generated indecomposable projective left $S$-module. Then $G(Q)$ is a finitely generated indecomposable projective left $R$-module and so ${G(Q) \cong Re_j}$ for some $1 \leq j \leq r$. Since $F(G(Q)) \cong Q$,  $Q \cong F(Re_j)$ for some $1 \leq j \leq r$. Thus $\lbrace F(Re_1), \cdots, F(Re_t)\rbrace$ is the complete set of non-isomorphic  finitely generated indecomposable projective left $S$-modules. Therefore $r$ is the cardinal number of the basic set of idempotents of $S$.  Now we show that  $q_R(l) = q_S(l)$ for each $1\leq l \leq r$.
Assume that  $\lbrace M_1, \cdots, M_t \rbrace$ is the complete set of non-isomorphic finitely generated indecomposable left $R$-modules. Then $\lbrace F(M_1), \cdots, F(M_t) \rbrace$ is the complete set of non-isomorphic  finitely generated indecomposable left $S$-modules. Let $1\leq l \leq r$ and $1\leq j \leq t$. Since  ${\rm Hom}_R(Re_l, {\rm top}(M_j)) \cong {\rm Hom}_S(F(Re_l), F({\rm top}(M_j)))$ and ${\rm End}_R(Re_l) \cong {\rm End}_S(F(Re_l))$,  $\ell\left(  _{{\rm End}_R(Re_l)}{\rm Hom}_R(Re_l, {\rm top}(M_j))   \right) = \ell\left( _{{\rm End}_S(F(Re_l))}{\rm Hom}_S(F(Re_l), F({\rm top}(M_j))) \right) $. Since by Lemma \ref{A2},
${\rm Hom}_S(F(Re_l), F({\rm top}(M_j))) \cong {\rm Hom}_S(F(Re_l), {\rm top}(F(M_j)))$, \\
 $\ell\left(  _{{\rm End}_R(Re_l)}{\rm Hom}_R(Re_l, {\rm top}(M_j))   \right) = \ell\left( _{{\rm End}_S(F(Re_l))}{\rm Hom}_S(F(Re_l), {\rm top}(F(M_j))) \right) $. Therefore by Lemma \ref{A1}, $q_R(l) = q_S(l)$.
 \end{proof}

\section{A characterization of left K\"othe rings}

We now give a characterization of left $k$-cyclic rings.

\begin{The}\label{A7}
Let $R$ be a ring and $k \in {\Bbb{N}}$. Then the following conditions are equivalent.
\begin{itemize}
\item[$(1)$] $R$ is a left $k$-cyclic ring.
\item[$(2)$] $R$ is a  left artinian ring and  for each finitely generated indecomposable  left $R$-module $M$, $c_i(top(M)) \leq kp_R(i)$ for each $1\leq i \leq m$, where $\lbrace e_1, \cdots, e_m\rbrace$ is a basic set of idempotents of $R$ and ${R\cong \bigoplus_{i=1}^m {(Re_i)}^{p_R(i)}}$.
\item[$(3)$] $R$ is a representation-finite ring and $q_R(i) \leq kp_R(i)$ for each $1\leq i \leq m$, where $\lbrace e_1, \cdots, e_m\rbrace$ is a basic set of idempotents of $R$ and ${R\cong \bigoplus_{i=1}^m {(Re_i)}^{p_R(i)}}$.
\end{itemize}
\end{The}
\begin{proof}
$(1) \Rightarrow (2).$ Assume that $R$ is a left $k$-cyclic ring. Then by \cite[Proposition 53.6]{Wi}, $R$ is a left artinian ring and so ${R\cong \bigoplus_{i=1}^m {(Re_i)}^{p_R(i)}}$, where $m \in {\Bbb{N}}$, each $p_R(i) \in {\Bbb{N}}$ and $\lbrace e_1, \cdots, e_m\rbrace$ is a basic set of idempotents of $R$. Let $M$ be a finitely generated indecomposable left $R$-module and ${1 \leq i \leq m}$. Then $M$ is $k$-generated and so by Proposition \ref{A4}, $c_i({\rm top}(M)) \leq kp_R(i)$.\\
$(2) \Rightarrow (3).$ It follows from \cite[Proposition 54.3]{Wi}.\\
$(3) \Rightarrow (1).$ Assume that $R$ is a  representation-finite ring and $q_R(i) \leq kp_R(i)$ for each $1\leq i \leq m$, where $\lbrace e_1, \cdots, e_m\rbrace$ is a basic set of idempotents of $R$ and ${R\cong \bigoplus_{i=1}^m {(Re_i)}^{p_R(i)}}$.  Let   $\lbrace N_1, \cdots, N_s \rbrace$ be the complete set of non-isomorphic finitely generated indecomposable left $R$-modules and $1 \leq i \leq m$.  Then ${c_i({\rm top}(N_j))  \leq kp_R(i)}$ for each $1 \leq j \leq s$ and so by Proposition \ref{A4},  $N_j$ is a $k$-generated left $R$-module. Let $M$ be a left $R$-module. Since $R$ is  representation-finite, by \cite[Propositions 53.6 and 54.3]{Wi}, $M$ is a direct sum of finitely generated indecomposable left $R$-modules. Thus $R$ is a left $k$-cyclic ring.
\end{proof}

As immediate consequence of Theorem \ref{A7}, we have the following result.

\begin{Cor} \label{A6}
The following conditions are equivalent for a ring $R$.
\begin{itemize}
\item[$(1)$] $R$ is a left K\"othe ring.
\item[$(2)$] $R$ is a  left artinian ring and  for each finitely generated indecomposable  left $R$-module $M$, $c_i(top(M)) \leq p_R(i)$ for each $1\leq i \leq m$, where $\lbrace e_1, \cdots, e_m\rbrace$ is a basic set of idempotents of $R$ and ${R\cong \bigoplus_{i=1}^m {(Re_i)}^{p_R(i)}}$.
\item[$(3)$] $R$ is a representation-finite ring and $q_R(i) \leq p_R(i)$ for each $1\leq i \leq m$, where $\lbrace e_1, \cdots, e_m\rbrace$ is a basic set of idempotents of $R$ and ${R\cong \bigoplus_{i=1}^m {(Re_i)}^{p_R(i)}}$.
\end{itemize}
\end{Cor}

A finitely generated indecomposable left $R$-module $M$ is called {\it  multiplicity-free top}  if composition factors of top$(M)$ are pairwise non-isomorphic. Also,  a ring $R$ is called {\it left multiplicity-free top} if every finitely generated indecomposable left $R$-module is multiplicity-free top. A ring $R$ is called {\it multiplicity-free top} if it is a left and  right multiplicity-free top ring (see \cite{Ringel2}).

\begin{Cor}
Let $R$ be a basic ring. Then $R$ is a left K\"othe ring if and only if $R$ is an artinian left multiplicity-free top ring.
\end{Cor}
\begin{proof}
It follows from Corollary \ref{A6} and \cite[Proposition 54.3]{Wi}.
\end{proof}
In the following, we show that Corollary \ref{A6} is a generalization of the K\"othe-Cohen-Kaplansky theorem and Theorem 3.1 of \cite{BGMS}. In fact, all known results related to the characterization of left K\"othe rings obtain from Corollary \ref{A6}.\\

 A left $R$-module $M$ is called {\it local} if it has a unique maximal submodule which contains any proper submodule of $M$. A ring $R$ is called {\it of left local type} if every finitely generated indecomposable left $R$-module is local (see \cite{Sumioka}). An idempotent  $e \in R$ is called {\it left {\rm (}resp., right{\rm )} semicentral}  if $Re = eRe$ (resp., $eR = eRe$)(see \cite{Kim}).

\begin{Pro}\label{KCK}
Let $R$ be a semiperfect ring that all primitive idempotents of $R$ are left semicentral. Then the following conditions are equivalent.
\begin{itemize}
\item[$(1)$] $R$ is a left multiplicity-free top ring.
\item[$(2)$] $R$ is of left local type.
\end{itemize}
\end{Pro}
\begin{proof}
$(1) \Rightarrow (2).$ Since $R$ is a semiperfect ring, $R = \bigoplus_{i=1}^nRe_i$  with $e_1 + \cdots+ e_n = 1_R$, where $e_1,\cdots, e_n$ are orthogonal primitive idempotents of $R$ and each $e_iRe_i$ is a local ring. Set $R_i= e_iRe_i$  for each $1 \leq i \leq n$. Since all primitive idempotents of $R$ are left semicentral,  $R = \bigoplus_{i=1}^nR_i$ is a basic ring. Therefore by \cite[Proposition 27.10]{Anderson},  $\lbrace Re_1/Je_1, \cdots ,Re_n/Je_n \rbrace$  is the complete set of non-isomorphic simple left $R$-modules. Let $M$ be a finitely generated indecomposable left $R$-module.  Since  $R$ is  left multiplicity-free top, ${\rm top}(M) = S_1 \oplus \cdots \oplus S_t$, where $t \leq n$ and  $ S_j \cong Re_j/Je_j$ for each $j$. On the other hand, since $R$ is a finite direct sum of the rings $R_i$, there exists $1 \leq i \leq n$ such that $M$ is a  left $R_i$-module.
 Let $1 \leq l \leq t$. Then  $R_jS_l=0$ for each $j \neq i$. It follows that $l=i$ and so  ${\rm top}(M) = S_i$. Therefore $R$ is of left local type.\\
$(2) \Rightarrow (1)$ is clear.
\end{proof}

Recall that a left $R$-module $M$ is called {\it uniserial} if its  submodules are linearly ordered by inclusion. Also, a  ring $R$ is called   {\it left {\rm (}resp., right{\rm )}  uniserial} if it is uniserial  as a left (resp., right) $R$-module (see \cite{Anderson}).

 \begin{The}\label{KcK} Let $R$ be a left artinian ring that all primitive idempotents of $R$ are left semicentral. If $R$ is a left multiplicity-free top ring, then $R$ is an artinian principal right ideal ring.
 \end{The}
\begin{proof} Assume that $R$ is a left multiplicity-free top ring. Then by Proposition \ref{KCK}, $R$ is of left local type. Since $R$ is  left artinian, $R = \bigoplus_{i=1}^nRe_i$ with $e_1 + \cdots+ e_n = 1_R$, where $e_1,\cdots, e_n$ are orthogonal primitive idempotents of $R$ and each $e_iRe_i$ is a local ring. Set $R_i= e_iRe_i$  for each $1 \leq i \leq n$. Since all primitive idempotents of $R$ are left semicentral, $R = \bigoplus_{i=1}^nR_i$, where each $R_i$ is a local ring.  It follows that each $R_i$ is of left local type.  Consequently, by \cite[Theorem 2.4]{Sing},  each $R_i$ is a right uniserial ring. On the other hand, since each $R_i$ is a left artinian ring of left local type, there is a finite upper bound for the lengths of finitely generated indecomposable modules in $R_i$-Mod. Thus by \cite[Proposition 54.3]{Wi}, each $R_i$ is an artinian ring. Therefore each $R_i$ is an artinian right uniserial ring. Consequently, by \cite[Proposition 56.3]{Wi}, $R$ is an artinian principal right ideal ring.
\end{proof}

As  consequences of Corollary \ref{A6}, Theorem \ref{KcK} and \cite{K}, we have the following results.

\begin{Cor}\label{KKK} Let $R$ be a left artinian ring that all primitive idempotents of $R$ are left semicentral. Then $R$ is a multiplicity-free top ring if and only if  $R$ is an artinian principal ideal ring.
\end{Cor}

\begin{Cor}\label{G2}{\rm(\cite[K\"othe-Cohen-Kaplansky Theorem]{K, CK})}
A commutative ring $R$ is a K\"othe ring if and only if $R$ is an artinian principal ideal ring.
\end{Cor}

\begin{Cor}\label{G1}{\rm(\cite[Theorem 3.1]{BGMS})}
Let $R$ be a ring in which all idempotents are central. If $R$ is a left K\"othe ring, then $R$ is an artinian principal right ideal ring.
\end{Cor}

The following example shows that there exists an artinian left multiplicity-free top local ring $R$ which is not principal left ideal ring.

\begin{Examp}\label{E1} {\rm
Let $F$ be a field which is isomorphic to the its proper subfield $\overline{F}$ such that ${\rm dim}(_{\overline{F}}F) = 2$ (for example let $F={\Bbb{Z}}_2(y)$, where ${\Bbb{Z}}_2(y)$ is the quotient field of polynomial ring ${\Bbb{Z}}_2[y]$ and let $\overline{F}={\Bbb{Z}}_2(y^2)$). Let $\alpha$ be the isomorphism from $F$ to $\overline{F}$ and $F[x; \alpha]$ be a skew polynomial ring with a usual polynomial addition and multiplication given by $\lambda x =x \alpha (\lambda)$ for each $\lambda \in F$.  Set $R:=F[x; \alpha] /{< x^2 >}$. Then $R$ is a local ring. Let ${\mathcal{M}}$ be the maximal ideal of $R$ and $\lbrace 1, a\rbrace$  be a basis for the vector space $F$ over $\overline{F}$. Therefore ${\mathcal{M}} =xR=Rx \oplus Rxa$ and ${\mathcal{M}}^2=0$. Set $Q=R/{\mathcal{M}}$. Consequently, ${\rm dim}(_Q{\mathcal{M}})=2$ and ${\rm dim}({\mathcal{M}}_Q) =1$. Hence by \cite[Proposition 3]{Ringel}, $R$ is an artinian left multiplicity-free top ring but it is not principal left ideal ring.}
\end{Examp}

The following example shows that the converse of Theorem \ref{KcK} and Corollary \ref{G1} are not true in general.

\begin{Examp}\label{E2}{\rm
Let $H$ be a division ring which is isomorphic to the its proper subdivision ring $\overline{H}$ such that  ${\rm dim}(_{\overline{H}}H) = 3$ (see \cite[Theorem]{Xue}). Let $\alpha$ be the isomorphism from $H$ to $\overline{H}$ and $H[x; \alpha]$  be a skew polynomial ring.  Set $R:=H[x; \alpha] /{< x^2 >}$.  Then $R$ is a local ring. Let ${\mathcal{M}}$ be the maximal ideal of $R$ and $\lbrace 1, a, b \rbrace$ be a basis for the vector space $H$ over $\overline{H}$. Then ${\mathcal{M}} =xR=Rx \oplus Rxa \oplus Rxb$ and ${\mathcal{M}}^2=0$. Thus  by \cite[Theorem 9]{Facc}, $R$ is a right uniserial ring and  ${\rm dim}(_Q{\mathcal{M}})=3$ and ${\rm dim}({\mathcal{M}}_Q) =1$, where $Q=R/{\mathcal{M}}$. It follows that  $\ell(R_R)=2$. Let $u, v, w$ be the linearly independent elements of $_Q{\mathcal{M}}$. The similar argument as in the proof of \cite[Lemma 3.1]{Ringel1} shows that  $T= (R \oplus R \oplus R)/D$, where $D = \lbrace (u\lambda, v\lambda, w\lambda)~|~ \lambda\in R \rbrace$ is an indecomposable right $R$-module with $\ell({\rm soc}(T_R))=2$. Therefore by \cite[Theorem B]{Sumioka},  $R$ is not of left local type. Hence by Proposition \ref{KCK}, $R$ is not a left multiplicity-free top ring but $R$ is an artinian principal right ideal ring. }
\end{Examp}

\section{A characterization of representation-finite rings}
Let ${\mathcal{U}}$ be a class of left $R$-modules and  $M$ be  a left $R$-module. Then $Tr({\mathcal{U}}, M) = \sum  \lbrace Im(h) ~ |~ h\in {\rm Hom}_R(U, M),~ U \in {\mathcal{U}}\rbrace $
is called  {\it trace of ${\mathcal{U}}$ in $M$}. An idempotent $e$ of $R$ is called {\it full idempotent} if $ReR=R$. We recall that for a full idempotent $e \in R$,  ${\rm Tr}(Re,R) = ReR=R$ and so $Re$ is a generator in $R-$Mod (see \cite[Exercise 13.10(1)]{Wi}).

\begin{The}\label{A9}
Let  $S$  be a basic ring and  $k \in {\Bbb{N}}$. Then  the following conditions are equivalent.
\begin{itemize}
\item[$(1)$] $S$ is a left $k$-cyclic ring.
\item[$(2)$] Any ring Morita equivalent to $S$ is left $k$-cyclic.
\item[$(3)$] Any ring $R$ which is Morita equivalent to $S$ is artinian and  for each indecomposable left $R$-module $M$,  $c_i({\rm top}(M))  \leq k$ for each $1 \leq i \leq m$, where $\lbrace e_1, \cdots, e_m\rbrace$ is a basic set of idempotents of $R$.
\item[$(4)$] For each full idempotent $e \in S$, $eSe$ is a  left $k$-cyclic ring.
\item[$(5)$] There exists a full idempotent $e \in S$ such that $eSe$ is a left $k$-cyclic ring.
\end{itemize}
\end{The}
\begin{proof}
Let $S$  be a basic ring. Then  $S = \bigoplus_{j=1}^tSf_j$, where $t \in {\Bbb{N}}$ and $\lbrace f_1, \cdots, f_t \rbrace$ is a basic set of idempotents of $S$.\\
$(1) \Rightarrow (2).$ Assume that $S$ is a left $k$-cyclic ring. Then by \cite[Proposition 53.6]{Wi}, $S$ is left artinian and so  there is a finite upper bound for the lengths of finitely generated indecomposable modules in $S$-Mod. Thus by \cite[Proposition 54.3]{Wi}, $S$ is a representation-finite ring. Let $R$ be Morita equivalent to $S$. Then $R$ is a representation-finite ring and so  ${R \cong \bigoplus_{i=1}^s {(Re_i)}^{p_R(i)}}$, where $s\in {\Bbb{N}}$,   each $p_R(i) \in {\Bbb{N}}$ and $\lbrace e_1, \cdots, e_s \rbrace$ is a basic set of idempotents of $R$. Therefore  by Proposition \ref{A5},  $t=s$ and $q_R(j) = q_S(j)$ for each $1 \leq j \leq s$.  Since $S$ is a basic left $k$-cyclic ring, by Theorem \ref{A7}, $q_S(j) \leq k$ for each $1 \leq j \leq s$. Hence by Theorem \ref{A7}, $R$ is a left $k$-cyclic ring.\\
$(2) \Rightarrow (3).$ Assume that any ring Morita equivalent to $S$ is left $k$-cyclic. Let  $R$ be Morita equivalent to $S$. Then $R$ is left $k$-cyclic. By \cite[Proposition 53.6]{Wi}, $R$ is left artinian and so  there is a finite upper bound for the lengths of finitely generated indecomposable modules in $R$-Mod. Thus by \cite[Proposition 54.3]{Wi}, $R$ is a representation-finite ring. Consequently, by \cite[Proposition 54.3]{Wi},  $R$ is an artinian ring. It follows that ${R\cong \bigoplus_{i=1}^m {(Re_i)}^{p_R(i)}}$, where $m \in {\Bbb{N}}$, each $p_R(i) \in {\Bbb{N}}$ and $\lbrace e_1, \cdots, e_m\rbrace$ is a basic set of idempotents of $R$. So by Theorem \ref{A5}, $t=m$ and  $q_R(i) =q_S(i)$ for each $1 \leq i \leq t$. Hence by Theorem \ref{A7}, $q_R(i) \leq k$ for each $1 \leq i \leq t$. Let $M$ be an indecomposable left $R$-module. Then $M$ is a $k$-generated module. Therefore $c_i({\rm top}(M)) \leq k$ for each $1 \leq i \leq t$.\\
 $(3) \Rightarrow (2).$  It follows from Theorem \ref{A7}.\\
 $(2) \Rightarrow (1)$ is clear.\\
$(1) \Rightarrow (4).$ Assume that $S$ is a left $k$-cyclic ring and $e$ is a full idempotent of $S$. Then $Se$ is a generator.  Hence by    \cite[Corollary 22.4]{Anderson}, $S \approx eSe$. Therefore by (2),  $eSe$ is a left $k$-cyclic ring.\\
$(4) \Rightarrow (5)$ is clear.\\
$(5) \Rightarrow (1).$ Assume that there exists a full idempotent $ 1_S \neq e \in S$ such that $eSe$ is a left $k$-cyclic ring. Then  $Se$ is a progenerator and by \cite[Propositions 53.6 and 54.3]{Wi}, $eSe$ is a representation-finite ring.  Hence by \cite[Corollary 22.4]{Anderson}, $eSe$  is Morita equivalent to $S$ via an equivalence $Se \otimes_{eSe} -: eSe{\rm -Mod} \rightarrow S{\rm -Mod}$. It follows that $S$ is a representation-finite ring. It is sufficient to show that every finitely generated indecomposable left $S$-module is $k$-generated. Let $Y$ be a finitely generated indecomposable left $S$-module. Then there exists a finitely generated indecomposable left $eSe$-module $X$ such that $Y\cong Se \otimes_{eSe} X$. Since $eSe$ is left $k$-cyclic,  there exists an epimorphism ${(eSe)}^k \rightarrow X$. Thus there exists an epimorphism $Se \otimes_{eSe} {(eSe)}^k \rightarrow Se \otimes_{eSe} X$. Since $Se \otimes_{eSe} eSe \cong Se$ as left $S$-module and $e$ is idempotent, there exists an epimorphism $S^k \rightarrow Se \otimes_{eSe} X$. Therefore $Y$ is a $k$-generated left $S$-module.
\end{proof}

\begin{Cor}\label{A8}
The following conditions are equivalent for a basic ring $S$.
\begin{itemize}
\item[$(1)$] $S$ is a left K\"othe ring.
\item[$(2)$] Any ring Morita equivalent to $S$ is left K\"othe.
\item[$(3)$] Any ring Morita equivalent to $S$ is an artinian left multiplicity-free top.
\item[$(4)$] For each full idempotent $ e \in S$, $eSe$ is a left K\"othe ring.
\item[$(5)$] There exists a full idempotent $e \in S$ such that $eSe$ is a left K\"othe ring.
\end{itemize}
\end{Cor}

 The following example shows that there exists a basic ring $S$ and an idempotent $e$ of $S$ such that $eSe$ is left K\"othe but $S$ is not a left K\"othe ring.

\begin{Examp}{\rm
Let $R$ be a simple artinian ring and $A$ be a basic finite dimensional algebra which is not left K\"othe and $R$ is not isomorphic to each direct summand of $A$. Then $S= A \oplus R$ is a basic ring which is not left K\"othe but $(0, 1)S(0, 1)$ is a left K\"othe ring.}
 \end{Examp}

\begin{Cor}\label{B3}
Let $R$ be a left K\"othe ring. Then there exists a positive integer $k$ such that every ring Morita equivalent to $R$ is left $k$-cyclic.
\end{Cor}
\begin{proof}
Assume that $R$ is a left K\"othe ring. Then $R$ is left artinian and so \\ ${R\cong \bigoplus_{i=1}^m {(Re_i)}^{p_R(i)}}$,  where $m \in {\Bbb{N}}$, each $p_R(i) \in {\Bbb{N}}$ and $\lbrace e_1, \cdots, e_m\rbrace$ is a basic set of idempotents of $R$. Set $e = e_1 + \cdots + e_m$, $S= eRe$ and $k= {\rm max}\lbrace p_R(i)~|~ 1 \leq i \leq m\rbrace$.  Thus  $S$ is a basic ring and by \cite[Proposition 27.10]{Anderson},  $P=eR$ generate all simple right $R$-modules. So by \cite[Proposition 17.9]{Anderson}, $P$ is a generator in Mod-$R$. It follows that there is a right $R$-module $R'$ such that $P^{(k)} \cong R \oplus R'$ and also by \cite[Corollaries 22.4 and 22.5]{Anderson}, $R \approx S$ via an equivalence $P\otimes_R -:R{\rm -Mod} \rightarrow S{\rm -Mod}$. Let $Y$ be a finitely generated indecomposable left $S$-module. Then there exists a finitely generated indecomposable left $R$-module $X$ such that $Y\cong P\otimes_R X$. Since $R$ is left K\"othe,  there exists an epimorphism $R \rightarrow X$ and so there exists an epimorphism $P\otimes_R R \rightarrow Y$. On the other hand, by \cite[Proposition 11.10]{Wi} and \cite[Proposition 4.5]{Anderson}, we have $S$-isomorphisms
\begin{center}
$_S{S^{(k)}} \cong {{\rm Hom}_R(P, P)}^{(k)} \cong {\rm Hom}_R(P^{(k)}, P) \cong {\rm Hom}_R(R \oplus R', P) \cong {\rm Hom}_R(R, P) \oplus {\rm Hom}_R(R', P) \cong P \oplus {\rm Hom}_R(R', P)$.
\end{center}
Consequently, there exists an epimorphism $S^{(k)} \rightarrow Y$. Therefore by  \cite[Propositions 53.6 and 54.3]{Wi},  $S$ is a left  $k$-cyclic ring. Thus by Theorem \ref{A9},  every ring Morita equivalent to $R$ is left $k$-cyclic.
\end{proof}

\begin{Rem}{\rm
Let ${R\cong \bigoplus_{i=1}^m {(Re_i)}^{p_R(i)}}$,  where $m \in {\Bbb{N}}$, each $p_R(i) \in {\Bbb{N}}$ and $\lbrace e_1, \cdots, e_m\rbrace$ is a basic set of idempotents of $R$. Let $k$ be a positive integer such that $k \leq p_R(i)$ for each $1 \leq i \leq m$. Assume that any ring which is Morita equivalent to $R$ is left $k$-cyclic. Then by Theorem \ref{A9}, $R$ is artinian and  for each indecomposable left $R$-module $M$,  $c_i({\rm top}(M))  \leq k$ for each $1 \leq i \leq m$. Therefore by Corollary \ref{A6}, $R$ is a left K\"othe ring. In fact, if  $k \leq p_R(i)$ for each $1 \leq i \leq m$, then the converse of Corollary \ref{B3} is true.
}
\end{Rem}

The following example shows that the converse of Corollary \ref{B3} is not true in general.

\begin{Examp}\label{B2}{\rm
Let $Q$ be the quiver
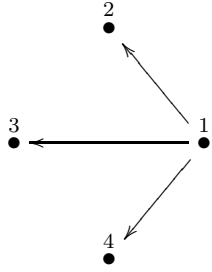
\begin{displaymath}
{\small \xymatrix{
 &  \overset{2}{\bullet}  \ar@{<-}[dr]      \\
\overset{3}{\bullet}  \ar@{<-}[rr] && \overset{1}{\bullet} \\
&  \overset{4}{\bullet}  \ar@{<-}[ur]
 }}
\end{displaymath}
and $A=KQ$ be the path algebra of $Q$ over an algebraically closed field $K$. We identify $A{\rm -mod} \approx{\rm rep}_K(Q)$.  Clearly $A$ is a basic representation-finite $K$-algebra. Let $M$ be the representation
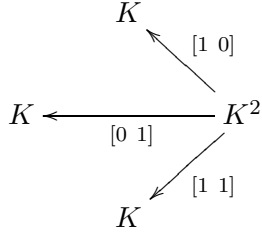
\begin{displaymath}
{\small \xymatrix{
 &  K  \ar@{<-}[dr]^{[1~ 0]}      \\
K   \ar@{<-}[rr]_{[0 ~ 1]} && K^2 \\
&  K  \ar@{<-}[ur]_{[1 ~ 1]}
 }}
\end{displaymath}
Then $M$ is a finitely generated indecomposable left $A$-module and it is easy to see that $c_1({\rm top}(M)) = 2$. Thus by Corollary \ref{A6}, $A$ is not left K\"othe. By using Theorem \ref{A7}, it is easy to see that $A$ is a left 2-cyclic ring. Therefore by Theorem \ref{A9}, every ring Morita equivalent to $A$ is left 2-cyclic.
}
\end{Examp}

It is known that the class of left K\"othe rings is a proper subclass of the class of representation-finite rings. In the following, we show that the class of representation-finite rings and the class of rings which are Morita equivalent to the left K\"othe rings are coincide.

\begin{Pro}\label{B4}
The following conditions are equivalent for a ring $R$.
\begin{itemize}
\item[$(1)$] $R$ is a representation-finite ring.
\item[$(2)$] There exists a basic ring $S$ and a positive integer $n$ such that ${\rm Mat}_n(S)$ is a left K\"othe ring and $R \approx {\rm Mat}_n(S)$.
\end{itemize}
\end{Pro}
\begin{proof}
$(1) \Rightarrow (2).$ Assume that $R$ is a representation-finite ring. Then there exists a basic ring $S$ such that $R \approx S$. Let $S = \bigoplus_{j=1}^rSf_j$, where $r \in {\Bbb{N}}$ and $\lbrace f_1, \cdots, f_r \rbrace$ is a basic set of idempotents of $S$. Set $d=q_S(1) + \cdots + q_S(r)$ and $T = {\rm Mat}_d(S)$. Then by \cite[Corollary 22.6]{Anderson}, $R \approx T$ and so $T$ is a representation-finite ring. It follows that  $T \cong \bigoplus_{k=1}^s{(Th_k)}^{p_{T}(k)}$, where $s \in {\Bbb{N}}$, each $p_{T}(k) \in {\Bbb{N}}$ and  $\lbrace h_1, \cdots , h_s \rbrace$ is a basic set of idempotents of $T$. Since $S$ is basic, $p_{T}(j) = d$ for each $1 \leq j \leq s$. On the other hand,  by Proposition \ref{A5},  $r = s$  and   $q_S(j) = q_{T}(j)$ for each $1 \leq j \leq r$. Consequently, $q_{T}(j) = q_S(j) \leq d = p_{T}(j)$ for each $1 \leq j \leq r$. Therefore by Corollary \ref{A6}, $T$ is a left K\"othe ring.\\
$(2) \Rightarrow (1).$ It follows from \cite[Propositions 53.6 and  54.3]{Wi}.
\end{proof}

\begin{Rem}{\rm
Let $R$ be a representation-finite ring which is not left K\"othe. Then there exists a basic ring $S$ and $n \in {\Bbb{N}}$ such that ${\rm Mat}_n(S)$ is a left K\"othe ring and $R \approx {\rm Mat}_n(S)$. In fact left K\"othe property is not a Morita invariant property.}
\end{Rem}

\begin{Pro}\label{B6}
Let $R$ be a ring and $n \in {\Bbb{N}}$. Then the following conditions are equivalent.
\begin{itemize}
\item[$(1)$] ${\rm{Mat}}_n(R)$ is a left k-cyclic ring.
\item[$(2)$] $R$ is a left $kn$-cyclic  ring.
\item[$(3)$] For each $m \geq n$, ${\rm{Mat}}_m(R)$ is a left k-cyclic ring.
\end{itemize}
\end{Pro}
\begin{proof}
$(1) \Rightarrow (2).$ Assume that ${\rm{Mat}}_n(R)$ is a left k-cyclic ring. Then by \cite[Propositions 53.6 and 54.3]{Wi}, ${\rm{Mat}}_n(R)$ is a representation-finite ring. Since $R \approx {\rm{Mat}}_n(R)$,  $R$ is a representation-finite ring. By \cite[Proposition 54.3]{Wi}, it is sufficient to show that every non-cyclic finitely generated indecomposable left $R$-module is $kn$-generated.  Let $M$ be a non-cyclic finitely generated indecomposable left $R$-module and $F:R{\rm -Mod} \rightarrow  {\rm{Mat}}_n(R){\rm -Mod}$ be an equivalence. Then $F(M)$ is a finitely generated indecomposable left $S$-module. Consequently,  $F(M)$ is a k-generated left $S$-module. Therefore by \cite[Example 17.23]{Lam2}, $M$ is a $kn$-generated left $R$-module.\\
$(2) \Rightarrow (3).$ Assume that $R$ is a left $kn$-cyclic  ring. Then by \cite[Propositions 53.6 and 54.3]{Wi}, $R$ is a representation-finite ring. Let $m$ be a positive integer such that $m \geq n$. Set  $T={\rm{Mat}}_m(R)$.  Since  $R \approx T$, $T$ is a representation-finite ring. By \cite[Theorem 54.3]{Wi}, it is sufficient to show that every finitely generated indecomposable left $T$-module is k-generated. Let $X$ be a finitely generated indecomposable left $T$-module and  $G:T{\rm -Mod} \rightarrow R{\rm -Mod}$ be an equivalence. Then $G(X)$ is a finitely generated indecomposable left $R$-module. It follows that $G(X)$ is a $kn$-generated module. Thus there exists an epimorphism $\alpha: R^{kn} \rightarrow G(X)$. Consequently, by \cite[Example 17.23]{Lam2}, there exists an epimorphism $S^k \rightarrow X$. Therefore $X$ is a k-generated left $S$-module.\\
$(3) \Rightarrow (1)$ is clear.
\end{proof}

Now we are in a position to give a characterization of representation-finite rings.

\begin{Cor}\label{D1}
The following conditions are equivalent for a ring $R$.
\begin{itemize}
\item[$(1)$] $R$ is a representation-finite ring.
\item[$(2)$] There exists a basic ring $S$ and a positive integer $n$ such that ${\rm Mat}_n(S)$ is a left K\"othe ring and $R \approx {\rm Mat}_n(S)$.
\item[$(3)$] There exists a positive integer $n$ such that  $R$ is a left n-cyclic ring.
\item[$(4)$] There exists a positive integer $n$ such that ${\rm{Mat}}_n(R)$ is a left K\"othe ring.
\item[$(5)$] There exists a positive integer $n$ such that for each $m \geq n$, ${\rm{Mat}}_m(R)$ is a left K\"othe ring.
\end{itemize}
\end{Cor}
\begin{proof}
$(1) \Leftrightarrow(2).$ It follows  from Proposition \ref{B4}.\\
 $(2) \Rightarrow (3).$ It follows from Corollary \ref{B3}.\\
 $(3) \Leftrightarrow (4) \Leftrightarrow (5).$ It follows from Proposition \ref{B6}.\\
 $(4) \Rightarrow (1)$ is clear.
\end{proof}

\section*{acknowledgements} The research of the first author was in part supported by a grant from Iran National Sciences Foundation: INSF and supported by a grant from IPM (No. 96160069). Also, the research of the second author was in part supported by a grant from IPM (No. 96170419).

\end{document}